\documentclass[12pt,reqno]{amsart}

\usepackage{amsthm}

\newtheorem{lemma}{Lemma}
\newtheorem{theorem}{Theorem}

\newtheorem{conjecture}{Conjecture}

\newtheorem*{theorem*}{Theorem}

\title{On a group ring identity related to the Alon-Jaeger-Tarsi conjecture}

\begin{document}

\author{J\'anos Nagy}

\address{Alfréd R\'enyi Institute of Mathematics.}

\email{janomo4@gmail.com}

\author{P\'eter P\'al Pach}

\address{HUN-REN Alfr\'ed R\'enyi Institute of Mathematics, Re\'altanoda utca 13--15., H-1053 Budapest,  Hungary;   \newline \hspace*{4mm}
MTA--HUN-REN RI Lend\"ulet ``Momentum'' Arithmetic Combinatorics Research Group, Re\'altanoda utca 13--15., H-1053 Budapest,  Hungary; \newline \hspace*{4mm} 
Department of Computer Science and Information Theory, Budapest University of Technology and Economics, M\H{u}egyetem rkp. 3., H-1111 Budapest, Hungary; \newline \hspace*{4mm}
Extremal Combinatorics and Probability Group (ECOPRO), Institute for Basic Science (IBS), Daejeon, South Korea.}
  
  \email{pachpp@renyi.hu}



\begin{abstract}
In this note we formulate a conjecture about two group ring identities and prove that it would imply the Alon-Jaeger-Tarsi conjecture.
\end{abstract}

\maketitle

\linespread{1.2}


\section{Introduction}

The Alon-Jaeger-Tarsi conjecture \cite{Alon-Tarsi, Jaeger}, first proposed in 1981  is the following:

\begin{conjecture}[Alon-Jaeger-Tarsi]\label{conj-AJT}
For any field $\mathbb{F}$ with $|\mathbb{F}|\geq 4$ and any nonsingular matrix $M$ over $\mathbb{F}$, there is a vector $x$ such that both $x$ and $Mx$ have only nonzero entries.
\end{conjecture}

Alon and Tarsi \cite{Alon-Tarsi} reached this conjecture while trying to generalize some simple properties of sparse graphs to more general matroids. 
Note that the conjecture is trivial when $\mathbb{F}$ is an infinite field. Alon and Tarsi \cite{Alon-Tarsi} settled the case when $|\mathbb{F}|\geq 4$ is a {\it proper} prime power using the polynomial method from \cite{Alon}. However, the case $\mathbb {F}=\mathbb{F}_p$ (with $p\geq 5$ being a prime) remained open. Recently, we \cite{NP} proved the conjecture for sufficiently large primes, namely, when $61<p\ne 79$.

Our proof made use of group ring identities and 
we proposed the following conjecture:

\begin{conjecture}\label{41}
Let $p > 3$ be a prime and let $M$ be a nonsingular $n \times n$ matrix over $\mathbb{F}_p$ with
row vectors $a_1, \dots,a_n$, finally, let the standard basis vectors of $G:=\mathbb{F}_p^n$ be $e_1,\dots,e_n$.

\noindent
Then the $\mathbb{Z}$ group ring identity 
$$\prod_{1 \leq i \leq n} (1 - g^{e_i}) \cdot   \prod_{1 \leq i \leq n} (1 - g^{a_i})  = 0 \in \mathbb{Z}[G]$$
follows from the mod $p$ group ring identity 
$$\prod_{1 \leq i \leq n} (1 - g^{e_i}) \cdot   \prod_{1 \leq i \leq n} (1 - g^{a_i})  = 0 \in \mathbb{F}_p[G].$$
\end{conjecture}

When $p$ is sufficiently large ($61<p\ne 79$), in \cite{NP} we could prove both Conjecture~\ref{conj-AJT}~and~Conjecture~\ref{41}. However, the cases when $5\leq p\leq 61$ or $p=79$ remained open. 
In this note we show that the Alon-Jaeger-Tarsi conjecture would follow from Conjecture~\ref{41} in the case of these ``small'' primes, as well.

\begin{theorem}\label{thm-1}
For every prime $p>3$ Conjecture~\ref{41} implies the Alon-Jaeger-Tarsi conjecture.
\end{theorem}


\section{Notations}\label{sec-not}

Throughout the paper we use the notation $[n]:=\{1,2,\dots,n\}$ and we denote by $[a,b]:=\{k\ |\ a\leq k\leq b,\ k\in\mathbb{Z} \}$  the set of integers between $a$ and $b$ included.

In the paper if we consider a group ring $R[G]$, where $R$ is an arbitrary ring and $G = \mathbb{F}_p^n$ for some $n \geq 1$ and $v \in \mathbb{F}_p^n$ is an arbitrary vector, then we use the multiplicative notation as follows. The \emph{group ring} $R[G]$ is the ring of formal expressions $\sum_{v\in G} r_v g^{v}$, where $r_v\in R$, and $g$ is a formal variable. (Here, the formal exponentiation $g^{v}$ is used to turn the additive structure of $G$ into a multiplicative one.) Addition and multiplication are defined in the natural way, that is, 
$$\left(\sum_{v\in G} r_v g^{v}\right)+\left(\sum_{v\in G} r'_v g^{v}\right)=\sum_{v\in G} (r_v+r_v') g^{v},$$ 
and 
$$\left(\sum_{v\in G} r_v g^{v}\right)\cdot\left(\sum_{v\in G} r'_v g^{v}\right)=\sum_{v\in G} \left(\sum_{w\in G} r_w r'_{v-w}\right)g^{v}.$$


\section{Proof of Theorem 1}\label{sec-small-primes}

In this section, we shed light on the structural importance of Conjecture~\ref{41} and demonstrate how the Alon-Jaeger-Tarsi (AJT) conjecture follows from the proposed equivalence of group ring identities. 

\medskip

\begin{proof}[Proof of Theorem 1]
For the sake of contradiction, let us assume that Conjecture~\ref{41} holds, but the Alon-Jaeger-Tarsi conjecture fails for some prime $p > 3$. 

\medskip

Let $n$ be the smallest dimension for which there exists a counterexample to the AJT conjecture, and let $M$ be  such a nonsingular $n\times n$ matrix (over $\mathbb{F}_p$) with row vectors $a_i$ ($i \in [n]$). Let $G := \mathbb{F}_p^n$ denote the underlying additive group.

\medskip

Based on the polynomial reformulation and the properties of counterexamples to the AJT conjecture \cite[Appendix B, Corollary 4]{NP}, we know that the following $\mathbb{Z}$ group ring identity must hold:
\begin{equation*}
\prod_{1 \leq j \leq n} (1 - g^{e_j}) \cdot \prod_{1 \leq j \leq n} (1 - g^{a_j}) = 0 \in \mathbb{Z}[G].
\end{equation*}
Consequently, by considering the coefficients modulo $p$, the corresponding identity in the $\mathbb{F}_p$ group ring also holds:
\begin{equation}\label{modp}
\prod_{1 \leq j \leq n} (1 - g^{e_j}) \cdot \prod_{1 \leq j \leq n} (1 - g^{a_j}) = 0 \in \mathbb{F}_p[G].
\end{equation}

\medskip

To achieve the desired contradiction through induction on the dimension, we first establish a reduction lemma that allows us to eliminate one of the factors.

\medskip

\begin{lemma}\label{lem:reduction}
Let $i\in [n]$ be a fixed index. Under the assumption that $M$ is a minimal counterexample, the following identity holds in the $\mathbb{Z}$ group ring:
\begin{equation*}
\prod_{\substack{1 \leq j \leq n,\\ j \neq i}} (1 - g^{e_j}) \cdot \prod_{1 \leq j \leq n} (1 - g^{a_j}) = 0 \in \mathbb{Z}[G].
\end{equation*}
\end{lemma}

\medskip

\begin{proof}
Let us write each row vector as $a_j = a'_j + a_{j, i} e_i$, where $a'_j$ belongs to the subgroup $G' := \langle e_1, \dots, e_{i-1}, e_{i+1}, \dots, e_n \rangle \cong \mathbb{F}_p^{n-1}$. 

\medskip

In the group ring $\mathbb{Z}[G]$, for each $j \in [n]$, the term $(1 - g^{a_j})$ can be expanded as:
\[ 1 - g^{a_j} = (1 - g^{a'_j}) + g^{a'_j}(1 - g^{a_{j,i}e_i}). \]
Since $(1 - g^{ke_i})$ is always divisible by $(1 - g^{e_i})$ in the group ring, we can write $1 - g^{a_j} = (1 - g^{a'_j}) + (1 - g^{e_i}) \cdot y_j$ for some elements $y_j \in \mathbb{Z}[G]$.

\medskip

Substituting these expressions into the mod $p$ identity \eqref{modp}, we obtain:
\begin{equation*}
B := \prod_{1 \leq j \leq n} (1 - g^{e_j}) \cdot \prod_{1 \leq j \leq n} ((1 - g^{a'_j}) + (1 - g^{e_i}) \cdot y_j) = 0 \in \mathbb{F}_p[G].
\end{equation*}

\medskip

The element $B$ can be viewed as a polynomial in the variable $(1 - g^{e_i})$. Since $(1 - g^{e_i})^p = 1 - g^{pe_i} = 0$ in $\mathbb{F}_p[G]$, we can write:
\[ 0 = B = \sum_{1 \leq k \leq p-1} b_k (1 - g^{e_i})^k, \]
where each coefficient $b_k$ lies in the group ring $\mathbb{F}_p[G']$. 

\medskip

By examining the linear term (the coefficient of $(1 - g^{e_i})^1$), and noting that the first product $\prod_{j} (1 - g^{e_j})$ already contains one factor of $(1 - g^{e_i})$, we find:
\[ b_1 = \prod_{\substack{1 \leq j \leq n,\\ j \neq i}} (1 - g^{e_j}) \cdot \prod_{1 \leq j \leq n} (1 - g^{a'_j}) \in \mathbb{F}_p[G']. \]
Since the expansion of $B$ must vanish, the linear independence of the powers $(1-g^{e_i})^k$ ($k=0,1\dots,p-1$) yields, in particular, that $b_1 = 0 \in \mathbb{F}_p[G']$.

\medskip

Now, our crucial claim is that this mod $p$ identity over $G'$ implies the corresponding identity over $\mathbb{Z}$:
\begin{equation}\label{eq2}
\prod_{\substack{1 \leq j \leq n,\\ j \neq i}} (1 - g^{e_j}) \cdot \prod_{1 \leq j \leq n} (1 - g^{a'_j}) = 0 \in \mathbb{Z}[G'].   
\end{equation}

\medskip

Notice that this implication is structurally similar to the statement of Conjecture~\ref{41}. However, we cannot obtain it directly, since the vectors $\{a'_j\}_{j \in [n]}$ do not form a basis of the subgroup $G' \cong \mathbb{F}_p^{n-1}$; indeed, there is one additional vector in this set.

\medskip

Assume to the contrary that the $\mathbb{Z}$ group ring identity does not hold:
\begin{equation*}
\prod_{\substack{1 \leq j \leq n,\\ j \neq i}} (1 - g^{e_j}) \cdot   \prod_{1 \leq j \leq n} (1 - g^{a'_j}) \neq 0 \in \mathbb{Z}[G'].
\end{equation*}
Based on this assumption and our original counterexample $M$ to the Alon-Jaeger-Tarsi conjecture, we will construct a counterexample to Conjecture~\ref{41} in a higher-dimensional setting.

\medskip

By reindexing the rows if necessary, we may assume without loss of generality that the subset $\langle a'_2, \dots, a'_n \rangle$ generates the subgroup $G'$.

\medskip

Let us consider a vector space $V'$ over $\mathbb{F}_p$ defined by the following basis vectors: 
\[ V' = \langle e_1, \dots, e_n , e_{j, k} \mid j \in [2n], k \in [n], k \neq i \rangle. \] 
Note that the dimension of this space is $\dim(V') = 2n^2 - n$. 

\medskip

We shall now construct two specific bases, $B_1$ and $B_2$, for the space $V'$. 

\noindent
To simplify our expressions, let us write $v_j: = \sum\limits_{\substack{1 \leq k \leq n, \\ k \neq i}} a'_{1, k} e_{j, k}$ for $j \in [2n]$, and let $w_j := \sum\limits_{1 \leq k \leq n} a_{j, k} e_{k}$ denote the original row vectors for $j \in [n]$.

\medskip

\noindent
First, we define $B_1$ as the following disjoint union:
\[ B_1 = B'_1 \cup \bigcup\limits_{1 \leq t \leq 2n} B'_{1, t}, \]
where the primary part is given by
\[ B'_1 = \left\{ e_j + v_j \ |\ j \in [n] \right\}, \]
and the auxiliary components for each $t \in [2n]$ are
\[ B'_{1, t} = \{ e_{t, k} \mid k \in [n], k \neq i \}. \]

\medskip

\noindent
\medskip

Similarly, we define the second basis $B_2$ of $V'$ as the disjoint union $B_2 = B'_2 \cup \bigcup\limits_{1 \leq t \leq 2n} B'_{2, t}$, where the primary components are
\[ B'_2= \left\{ \sum\limits_{1 \leq k \leq n} a_{j-n, k} e_{k} + v_j \ |\ j \in [n+1, 2n] \right\}, \]
and the auxiliary parts for each $t \in [2n]$ are given by
\[ B'_{2, t} = \left\{ \sum\limits_{\substack{1 \leq k \leq n, \\ k \neq i}} a'_{\ell, k} e_{t, k} \ \Big|\ \ell \in [2, n] \right\}. \]

\medskip

\noindent
It can be easily checked that $B_1$ and $B_2$ are indeed two bases of the vector space $V'$.

\medskip

First, we claim that the following identity holds in the mod $p$ group ring:
\[ \prod_{b \in B_1} (1-g^b) \cdot \prod_{b \in B_2} (1-g^b) = 0 \in \mathbb{F}_p[V']. \]

\medskip

\noindent
For each index $t \in [2n]$, let us introduce the following notation:
\[ U_t := \prod\limits_{b \in B'_{1, t}} (1-g^b) \cdot \prod\limits_{b \in B'_{2, t}} (1-g^b). \]

\medskip

\noindent
By the properties of the vectors $a'_1, \ldots, a'_n, e_1, \ldots, e_{i-1}, e_{i+1}, \ldots, e_n$, we know that for each $t \in [2n]$, the identity $U_t \cdot (1 - g^{v_t}) = 0 \in \mathbb{F}_p[V']$ is satisfied. (This identity indeed holds, since the vector configuration occurring in $U_t \cdot (1 - g^{v_t})$ is equivalent to the vector configuration from \eqref{eq2}.)

\medskip

Observe that for any $j \in [n]$, we can decompose the terms as:
\[ 1 - g^{e_j + v_j} = (1 - g^{e_j}) + (1- g^{v_j}) z_j, \]
where $z_j=g^{e_j} \in \mathbb{F}_p[V']$.

\medskip

Similarly, for each $j \in [n]$, we can write:
\[ 1 - g^{w_j + v_{j+n}} = (1 - g^{w_j}) + (1- g^{v_{j+n}}) r_j, \]
where $r_j=g^{w_j} \in \mathbb{F}_p[V']$.

\medskip

It follows from these equations and the distributive properties of the group ring that:
\begin{equation*}
 \prod_{b \in B_1} (1-g^b) \cdot \prod_{b \in B_2} (1-g^b) = \prod_{1 \leq t \leq 2n} U_t \cdot \prod_{1 \leq j \leq n}(1 - g^{e_j}) \cdot \prod_{1 \leq j \leq n}(1 - g^{w_j}) \in \mathbb{F}_p[V'].
\end{equation*}

\medskip

\noindent
On the other hand, since $M$ is a counterexample to the Alon-Jaeger-Tarsi conjecture, we know that 
\[ \prod_{1 \leq j \leq n}(1 - g^{e_j}) \cdot \prod_{1 \leq j \leq n}(1 - g^{w_j}) = 0 \in \mathbb{F}_p[V'], \]
which indeed yields that
\[ \prod_{b \in B_1} (1-g^b) \cdot \prod_{b \in B_2} (1-g^b) = 0 \in \mathbb{F}_p[V']. \]

\medskip

\noindent
\emph{Next, we will prove that, on the other hand, we have: 
\[ \prod_{b \in B_1} (1-g^b) \cdot \prod_{b \in B_2} (1-g^b) \neq 0 \in \mathbb{Z}[V']. \] }

\medskip

For this, we have to show that there exists a vector $x \in V'$ such that $\langle x, b \rangle \neq 0$ for every $b \in B_1 \cup B_2$, where $\langle \cdot, \cdot \rangle$ denotes the standard scalar product corresponding to the basis vectors $e_1, \dots, e_n$ and $e_{j, k}$.

\medskip

Let us fix the values of $x_1, \dots, x_n$ in an arbitrary way. We show that we can extend these with values of $x_{j, k}$ in a proper way such that the conditions hold.

\medskip

Indeed, we obtain the following conditions for the values $x_{j,k}$ for some $c_j\in\mathbb{F}_p$:\\
for each $j\in [2n]$
\[ x_{j,k}\ne 0 \quad \text{for each } k\in[n],\ k\ne i \leq n,
\]
\[ \sum_{k \neq i} a'_{\ell, k} x_{j, k} \neq 0 \quad \text{for each } 2 \leq \ell \leq n,\]
\[ \sum_{k \neq i} a'_{1, k} x_{j, k} \neq c_j.  \]

\medskip

Now, from our assumptions on the vectors $a'_1, \dots, a'_n, e_1, \dots, e_{i-1}, e_{i+1}, \dots, e_n$, by using \cite[Lemma~17]{Szegedy} we can find a vector $y \in \langle e_{j, k} \mid k \in [n], k \neq i \rangle$ such that the first two types of conditions hold and 
\[ \sum\limits_{1 \leq k \leq n, k \neq i} a'_{1, k} y_{j, k} \neq 0. \] 

\medskip

If $c_j = 0$, then we can choose $x_{j, k} = y_{j, k}$, and if $c_j \neq 0$, then we can choose $x_{j, k} = \lambda_j \cdot y_{j, k}$ with some appropriate $0 \neq \lambda_j \in \mathbb{F}_p$.

\medskip

Thus, we have established that
\[ \prod\limits_{b \in B_1} (1-g^b) \cdot \prod\limits_{b \in B_2} (1-g^b) \neq 0 \in \mathbb{Z}[V'], \]
which directly contradicts our assumption that Conjecture~\ref{41} holds.

\medskip

We have therefore proven that 
\[ \prod_{\substack{1 \leq j \leq n,\\ j \neq i}} (1 - g^{e_j}) \cdot   \prod_{1 \leq j \leq n} (1 - g^{a'_j}) = 0 \in \mathbb{Z}[G']. \]

\medskip

\noindent
Let us define the product:
\[ Q := \prod_{\substack{1 \leq j \leq n,\\ j \neq i}} (1 - g^{e_j}) \cdot   \prod_{1 \leq j \leq n} (1 - g^{a_j}). \] 

\medskip

Observe that $Q \in \langle 1 - g^{e_i} \rangle \subseteq \mathbb{Z}[G]$, where $\langle 1 - g^{e_i} \rangle$ is the principal ideal generated by the element $1 - g^{e_i} \in \mathbb{Z}[G]$. This follows from identity~\eqref{eq2} by using the projection $G\to G'$ (in the direction of $e_i$). 
Consequently, we may write $Q = q \cdot (1 - g^{e_i})$.

\medskip

Now, observe that we have $Q \cdot (1 - g^{e_i}) = 0$, which means that $q \cdot (1 - g^{e_i})^2 = 0$.

\medskip

However, in the group ring $\mathbb{Z}[G]$, the kernel of the multiplication by the element $(1 - g^{e_i})^2$ is identical to the kernel of the multiplication by the element $(1 - g^{e_i})$.

\medskip

Indeed, for an element $\sum\limits_{v\in G} c_vg^v$, both the equation $(\sum\limits_{v\in G} c_vg^v)(1-g^{e_i})=0$ and the equation $(\sum\limits_{v\in G} c_vg^v)(1-g^{e_i})^2=0$ is equivalent to the condition that $c_v=c_{v-e_i}$ for every $v\in G$.

\end{proof}

\bigskip

\medskip

Now, let us return to the proof of Theorem~\ref{thm-1}. We know that the following identity holds in the $\mathbb{F}_p$ group ring:
\begin{equation}\label{eq1}
\prod_{1 \leq j \leq n} (1 - g^{e_j}) \cdot   \prod_{1 \leq j \leq n} (1 - g^{a_j})  = 0 \in \mathbb{F}_p[G].
\end{equation}

\medskip

\noindent
For an element $x = \sum\limits_{v \in \mathbb{F}_p^n} x_v g^v \in \mathbb{F}_p[G]$ and an index $i \in [n]$, let us define the operation:
\[ \partial_i(x) := \sum\limits_{v \in \mathbb{F}_p^n} x_v \cdot v_i \cdot g^v \in \mathbb{F}_p[G]. \]

\medskip

\noindent
It is easy to check that $\partial_i$ satisfies the property $\partial_i(x \cdot y) = \partial_i(x) \cdot y + x \cdot \partial_i(y)$ if $x, y \in \mathbb{F}_p[G]$.

\medskip

\noindent
If we apply the partial derivative $\partial_i$ to \eqref{eq1}, we obtain that
\begin{equation*}
  g^{e_i} \prod_{\substack{1 \leq j \leq n ,\\ j \neq i}} (1 - g^{e_j}) \cdot \prod_{1 \leq j \leq n} (1 - g^{a_j})  + \sum_{1 \leq k \leq n} a_{k, i}g^{a_k} \cdot \prod_{1 \leq j \leq n} (1 - g^{e_j}) \cdot   \prod_{\substack{1 \leq j \leq n,\\ j \neq k}} (1 - g^{a_j}) = 0 \in \mathbb{F}_p[G].
\end{equation*}

\medskip

\noindent
On the other hand, by our main lemma (Lemma~\ref{lem:reduction}), we know that the first term vanishes:
\[ \prod_{\substack{1 \leq j \leq n ,\\ j \neq i}} (1 - g^{e_j}) \cdot \prod_{1 \leq j \leq n} (1 - g^{a_j}) = 0 \in \mathbb{F}_p[G], \]
so we obtain for every $i \in [n]$ that:
\begin{equation*}
 \sum_{1 \leq k \leq n} a_{k, i}g^{a_k} \cdot \prod_{1 \leq j \leq n} (1 - g^{e_j}) \cdot   \prod_{\substack{1 \leq j \leq n,\\ j \neq k}} (1 - g^{a_j}) = 0 \in \mathbb{F}_p[G].
\end{equation*}

\medskip

\noindent
Since the matrix $M$ is invertible, its columns span the entire space, which implies that
\begin{equation*}
g^{a_k}\prod_{1 \leq j \leq n} (1 - g^{e_j}) \cdot   \prod_{2 \leq j \leq n} (1 - g^{a_j}) = 0 \in \mathbb{F}_p[G].
\end{equation*}

\medskip

Therefore, $u:=\prod_{1 \leq j \leq n} (1 - g^{e_j}) \cdot   \prod_{2 \leq j \leq n} (1 - g^{a_j}) = 0 \in \mathbb{F}_p[G]$, since $g^{a_k}$ is invertible.

\medskip

We may assume (by reindexing, if necessary) that the determinant of the matrix obtained from $M$ by deleting its first row and first column is nonzero.

\medskip

\noindent
Let us write $a_j = a_{j, 1} e_1 + a'_j$, then $a'_2, \ldots, a'_n$ is a basis of the subgroup $H := \langle e_2, \dots, e_n\rangle$.

\medskip

\noindent
Since for $j \in [2, n]$ we have the expansion $1 - g^{a_j} = (1 - g^{a'_j}) + (1-g^{e_1}) y_j$, we obtain that
\begin{equation*}
u = \prod_{1 \leq j \leq n} (1 - g^{e_j}) \cdot \prod_{2 \leq j \leq n} \left((1 - g^{a'_j}) + (1-g^{e_1}) y_j\right) = 0 \in \mathbb{F}_p[G].
\end{equation*}

\medskip

\noindent
We may write $0 = u = \sum\limits_{1 \leq k \leq p-1} (1 - g^{e_1})^k u_k$, where each $u_k$ belongs to $\mathbb{F}_p[H]$. 

\medskip

\noindent
On the other hand, by examining the coefficient of the first power $(1-g^{e_1})$, we find that
\[ u_1 = \prod\limits_{2 \leq j \leq n} (1 - g^{e_j}) \cdot \prod\limits_{2 \leq j \leq n} (1 - g^{a'_j}) = 0 \in \mathbb{F}_p[H]. \]

\medskip

\noindent
Hence, finally, we obtain the identity in the lower-dimensional group ring:
\begin{equation*}
\prod_{2 \leq j \leq n} (1 - g^{e_j}) \cdot \prod_{2 \leq j \leq n} (1 - g^{a'_j}) = 0 \in \mathbb{F}_p[H].
\end{equation*}

\medskip

\noindent
However, by using again the statement of Conjecture~\ref{41} for the nonsingular $(n-1)\times(n-1)$ submatrix (whose rows are $a_2',\dots,a_n'$), we would obtain a counterexample to the Alon-Jaeger-Tarsi conjecture with a smaller dimension than $n$. This contradicts our original choice of $n$ as the minimal dimension for a counterexample.

\medskip

\noindent
This contradiction finishes the proof of Theorem~\ref{thm-1}.
\end{proof}

\section{Acknowledgements} 
Both authors were supported by the National Research, Development and Innovation Office NKFIH (Excellence program, Grant Nr. 153829). PPP was also partially supported by the Institute for Basic Science (IBS-R029-C4).

\end{document}